\documentclass[11pt]{article}
\usepackage[margin=1in]{geometry}
\usepackage{amsmath}
\usepackage{amsfonts}
\usepackage{color}
\usepackage{latexsym}
\usepackage[colorlinks]{hyperref}
\usepackage{epsfig}
\usepackage{float}
\usepackage{array}
\usepackage{amssymb,amsthm}
\usepackage{algorithm,algorithmic}

\newcommand*{\QEDA}{\hfill\hbox{\vrule width1.0ex height1.0ex}}

\usepackage{url}

\usepackage{amsmath}
\usepackage{amsfonts}
\usepackage{latexsym}
\usepackage{amssymb}

\newtheorem{thm}{Theorem}[section]
\newtheorem{theorem}[thm]{Theorem}

\newtheorem{lemma}[thm]{Lemma}

\newtheorem{proposition}[thm]{Proposition}

\newtheorem{remark}[thm]{Remark}

\newcommand{\beq}{\begin{equation}}
\newcommand{\eeq}{\end{equation}}
\newcommand{\beqa}{\begin{eqnarray}}
\newcommand{\eeqa}{\end{eqnarray}}
\newcommand{\beqas}{\begin{eqnarray*}}
\newcommand{\eeqas}{\end{eqnarray*}}
\newcommand{\bi}{\begin{itemize}}
\newcommand{\ei}{\end{itemize}}

\newcommand{\vgap}{\vspace{.1in}}
\newcommand{\lessgap}{\vspace{-.1in}}

\newcommand{\nn}{\nonumber}

\newcommand{\R}{\mathbb{R}}
\newcommand{\E}{\mathbb{E}}
\newcommand{\Prob}{\mathbb{P}}

\newcommand{\cL}{{\cal L}}

\newcommand{\lam}{{\lambda}}

\newcommand{\inner}[2]{\langle #1,#2\rangle}

\newcommand{\argmin}{\mathrm{argmin}\,}

\newcommand{\dom}{\mathrm{dom}\,}

\begin{document}
	\title{Variance Reduction and Low Sample Complexity in\\ Stochastic Optimization via Proximal Point Method}
	\date{February 14, 2024 (1st revision: October 3, 2025; 2nd revision: December 21, 2025)}
	
	\author{
		Jiaming Liang \thanks{Goergen Institute for Data Science and Artificial Intelligence (GIDS-AI) and Department of Computer Science, University of Rochester, Rochester, NY 14620 (email: {\tt jiaming.liang@rochester.edu}). This work was partially supported by GIDS-AI seed funding and AFOSR grant FA9550-25-1-0182.
		}
	}
	\maketitle
	\maketitle
	
	\begin{abstract}
		
		High-probability guarantees in stochastic optimization are often obtained only under strong noise assumptions such as sub-Gaussian tails. 
		We show that such guarantees can also be achieved under the weaker assumption of bounded variance by developing a stochastic proximal point method. 
		This method combines a proximal subproblem solver, which inherently reduces variance, with a probability booster that amplifies per-iteration reliability into high-confidence results. 
		The analysis demonstrates convergence with low sample complexity, without restrictive noise assumptions or reliance on mini-batching.\\

		\\
		
	\end{abstract}
	
	\section{Introduction}
	
	A key challenge in stochastic optimization is obtaining high-confidence guarantees, which are crucial in practice but typically require restrictive assumptions such as sub-Gaussian noise.  
	In this paper, we address this challenge by introducing a stochastic proximal point method (SPPM) for solving the stochastic convex composite optimization problem
	\begin{equation}\label{eq:ProbIntro}
		\phi_{*}:=\underset{x \in \R^d}\min \left\{\phi(x):=f(x)+h(x)\right\},
	\end{equation}
	where $f(x)=\mathbb{E}_{\xi}[F(x,\xi)]$ is the expectation of a stochastic function $F(x,\xi)$. 
	We analyze the sample complexity of SPPM under the following assumptions: $h$ is closed, convex, and admits an efficiently computable proximal mapping; $f$ is closed and strongly convex; and the stochastic gradient oracle of $f$ is unbiased with bounded variance.

	Under the above assumptions, standard results on stochastic approximation (SA) methods \cite{nemirovskij1983problem,polyak1992acceleration,ghadimi2013optimal,liang2023single} provide non-asymptotic convergence guarantees in expectation. Specifically, for some tolerance $\varepsilon>0$, one can obtain an $\varepsilon$-solution $x\in \dom h$ satisfying $\E[\phi(x)]-\phi_* \le \varepsilon$ within ${\cal O}(1/\varepsilon)$ queries to the stochastic gradient oracle of $f$. High-probability guarantees of the form $\Prob(\phi(x)-\phi_* \le \varepsilon) \ge 1-p$ for given $\varepsilon>0$ and $p\in (0,1)$ can also be derived using Markov's inequality. However, this yields a sample complexity of ${\cal O}(1/(\varepsilon p))$, which can be unsatisfactory when $p$ is small.
	Without modifying the algorithm, prior works \cite{nemjudlannem09,juditsky2014deterministic,lan2012optimal,ghadimi2012optimal,ghadimi2013optimal} improve the dependence of $p$ from ${\cal O}(1/p)$ to ${\cal O}(\log(1/p))$ by assuming a stronger condition on the stochastic gradient noise, namely a sub-Gaussian distribution. This assumption, however, is more restrictive than the bounded variance condition considered in this paper. 
	
	Recent papers \cite{nazin2019algorithms,hsu2016loss,davis2021low} employ gradient clipping or classical probability tools, such as robust distance estimation (RDE) by \cite{nemirovskij1983problem}, to reduce the dependence of sample complexity on $p$ without imposing restrictive assumptions on stochastic gradient noise.
	In particular, \cite{davis2021low} proposes an interesting framework that incorporates an arbitrary SA method and RDE into the proximal point method (PPM). For strongly convex and smooth problems, the condition number of the proximal subproblem improves with smaller prox stepsizes, leading to sample complexity of with only logarithmic dependence on $1/p$.
	While conceptually appealing, the approach in \cite{davis2021low} raises several practical questions. First, the method is described in terms of a generic minimization oracle within PPM, leaving open which specific SA method might work best in practice. Second, the algorithm relies on a sequence of decreasing prox stepsizes that depend on the strong convexity parameter $\mu$, which is often unknown or difficult to estimate. Designing adaptive or universal methods that achieve optimal iteration complexity without prior knowledge of $\mu$ remains an active line of research, even in deterministic settings \cite{aujol2023parameter,nesterov2013gradient,lan2023optimal}. Finally, although their analysis also improves the dependence of sample complexity on the condition number compared with \cite{hsu2016loss}, this improvement similarly relies on access to $\mu$. Thus, while the framework in \cite{davis2021low} offers valuable theoretical insights, its practical applicability may depend on further advances in adaptive methods.
	
	\paragraph{SPPM in a nutshell.}
	Our algorithm, SPPM, builds on PPM with a \emph{constant} prox stepsize $\lambda>0$. 
	At iteration $k$, given the current prox–center $\bar z_{k-1}$, we consider the proximal subproblem
	\begin{equation}\label{eq:prox}
		\hat z_k:=\underset{x\in \R^d} \argmin \left\{\phi(x) + \frac{1}{2\lam}\|x-\bar z_{k-1}\|^2\right\}.
	\end{equation}
	SPPM does not require solving \eqref{eq:prox} exactly. Instead, it queries a \emph{proximal subproblem solver} (PSS) $n$ times to obtain statistically independent approximate solutions, and then applies a \emph{probability booster} (PB) that selects one candidate. 
	The booster is designed so that, with high probability, the selected point lies within a prescribed accuracy of the exact proximal point $\hat z_k$. 
	This selected point then becomes the next prox–center, i.e., $\bar z_k\gets \text{PB}(\{\text{PSS outputs}\})$, and the process repeats.
	
	Our analysis establishes high-probability guarantees for a PPM scheme with a \emph{fixed} $\lambda$. 
	The main difficulty is that the stochastic errors incurred in solving each proximal subproblem do not naturally diminish over time; instead, they can accumulate across iterations. 
	This requires a careful argument to show that the error remains controlled and does not overwhelm the contraction effect of the proximal operator. 
	By contrast, the high-confidence framework of \cite{davis2021low} (proxBoost) relies on a \emph{geometrically decaying} prox stepsize rule, which progressively suppresses noise across stages and simplifies the concentration analysis. 
	However, geometric decay has drawbacks: (i) it introduces stage lengths and decay factors that must be scheduled or tuned in advance, often depending on the problem horizon; and (ii) once the stepsize becomes very small, progress can stagnate due to overly conservative steps. 
	Our results demonstrate that one can obtain high-probability convergence \emph{without} resorting to geometric decaying stepsizes, by combining a constant-$\lambda$ PPM with a probability booster that carefully controls the stochastic error at each iteration.
	
	\paragraph{Contributions.}
	Our contributions build on two crucial subroutines, namely PSS and PB.  
	PSS returns inexact proximal points whose distribution has \emph{reduced variance} relative to raw stochastic gradient steps, which reflects an intrinsic variance-reduction effect of proximal regularization. 
	PB further amplifies this benefit by aggregating $n$ independent PSS calls and selecting a statistically reliable candidate, yielding a point that is sufficiently close to $\hat z_k$ with probability at least $1-p$ while requiring only $O(\log(1/p))$ extra samples. 
	Combining these ideas, we propose SPPM, a stochastic proximal point scheme with a constant prox stepsize $\lam$ (with a lower bound), together with the implementable PSS and PB oracles. 
	Under only a bounded-variance noise assumption, we establish a high-probability convergence guarantee with failure probability at most $p$ and a sample complexity that scales $O(\log(1/p))$; see the main results of the paper, namely Theorem~\ref{thm:main} (high-probability guarantee) and Theorem~\ref{thm:sample} (low sample complexity). 
	
	\paragraph{Organization of the paper.}
	Subsection~\ref{subsec:DefNot} introduces the basic notation and definitions used throughout the paper. 
	Section~\ref{sec:SPPM} formally states the optimization problem \eqref{eq:ProbIntro} and presents SPPM. 
	Section~\ref{sec:PSS} describes the PSS oracle and analyzes its properties. 
	Section~\ref{sec:PPM} develops the proximal point analysis that serves as the backbone of our results. 
	Section~\ref{sec:PB} presents and analyzes the PB oracle. 
	Section~\ref{sec:result} states the main results of the paper, namely a high-confidence convergence guarantee and a low sample complexity bound. 
	Section~\ref{sec:conclusion} provides concluding remarks and directions for future work. 
	Appendix~\ref{sec:tech} collects useful technical results, while Appendices~\ref{sec:sts} and~\ref{sec:rge} give the details and analysis of auxiliary oracles used in the paper.

	\subsection{Basic notation and definitions} \label{subsec:DefNot}
	
	Let $\R$ denote the set of real numbers.
	Let $\R^d$ denote the standard $d$-dimensional Euclidean space equipped with inner product and norm denoted by $\left\langle \cdot,\cdot\right\rangle $
	and $\|\cdot\|$, respectively. 
	Let $\log(\cdot)$ denote the natural logarithm. Let ${\cal O}$ denote the standard big-O notation. The notation $\tilde {\cal O}$ is used similarly, but allows additional logarithmic factors.
	
	For a given function $\psi: \R^d\rightarrow (-\infty,+\infty]$, let $\dom \psi:=\{x \in \R^d: \psi (x) <\infty\}$ denote the effective domain of $\psi$.
	We say $\psi$ is proper if $\dom \psi \ne \emptyset$.
	A proper function $\psi: \R^d\rightarrow (-\infty,+\infty]$ is $\mu$-strongly convex for some $\mu > 0$ if
	\[
	\psi(t x+(1-t) y)\leq t \psi(x)+(1-t)\psi(y) - \frac{\mu}{2}\|x-y\|^2
	\]
	for every $t \in [0,1]$ and $x, y \in \dom \psi$.
	The \emph{subdifferential} of $ \psi $ at $x \in \dom \psi$ is denoted by
	\beq \label{def:subdif}
	\partial \psi (x):=\left\{ s \in\R^d: \psi(y)\geq \psi(x)+\left\langle s,y-x\right\rangle, \, \forall y\in\R^d\right\}.
	\eeq

	\section{Stochastic Proximal Point Method and Main Results}\label{sec:SPPM}
	
	This section describes the assumptions made on problem \eqref{eq:ProbIntro}, presents the main algorithm SPPM to solve \eqref{eq:ProbIntro}, and gives an overview of the main results of the paper.
	
	Let $\Xi$ denote the support of random vector $\xi$ and assume that the following conditions on \eqref{eq:ProbIntro} hold:
	
	\begin{itemize}
		\item[(A1)] for almost every $\xi \in \Xi$, there exist a stochastic function oracle $F(\cdot,\xi) :\dom h \to \R$ and
		a stochastic gradient
		oracle $s(\cdot,\xi):\dom h \to \R^d$ satisfying
		\[
		f(x) = \E[F(x,\xi)], \quad \nabla f(x) = \E[s(x,\xi)] \in \partial f(x), \quad \forall x \in \dom h;
		\]
		\item[(A2)] for every $x \in \dom h$, we have $\E[\|s(x,\xi)-\nabla f(x)\|^2] \le \sigma^2$;
		\item[(A3)] for every $x,y \in \dom h$,
		\[
		\|\nabla f(x)- \nabla f(y)\|\le L \|x-y\|;
		\]
		\item[(A4)] $f$ is $\mu$-strongly convex, $h$ is convex, and $\dom h \subset \dom f$;
		\item[(A5)] $\dom h$ is bounded with diameter $D>0$.
	\end{itemize}
	
	It is well-known that Assumptions (A3) and (A4) imply that for every $x, y \in \dom h$,
	\begin{equation}\label{ineq:twoside}
		\frac{\mu}{2}\|y-x\|^2 \le  f(y) - f(x) - \inner{\nabla f(x)}{y-x} 
		\le \frac{L}{2}\|y-x\|^2.
	\end{equation}
	
	Now, we formally present SPPM in Algorithm~\ref{alg:PPM}. SPPM relies on two key oracles, namely PSS and PB, which are given and analyzed in Sections \ref{sec:PSS} and \ref{sec:PB}, respectively. Step 1 of Algorithm \ref{alg:PPM} repeatedly calls PSS for $n$ times to generate independent pairs $\{z_k^j,w_k^j\}_{j=1}^n$ and each pair satisfies a guarantee of low probability. Among the $n$ pairs output by PSS, Step 2 calls PB to select one candidate so that the same guarantee holds at a high confidence level.

	\begin{algorithm}[H]
		\caption{Stochastic Proximal Point Method, SPPM$(\bar z_0,\alpha,\lam,n,q,I,K)$}
		\begin{algorithmic}
			\STATE \textbf{Input:} Initial point $\bar z_0\in \dom h $, scalars $\alpha\in(0,1)$ and $\lam>0$, and integers $n, q, I, K \ge 1$.
			\FOR{$k = 1, \ldots, K$}
			\STATE {\bf Step 1.} Generate independent pairs $(z_k^1,w_k^1), \ldots, (z_k^n, w_k^n)$ by calling $n$ times the oracle PSS$(\bar z_{k-1},\alpha,\lam,I)$;
			\STATE {\bf Step 2.} Generate $(\bar z_k, \bar w_k)$ by calling the oracle PB$(\{(z_k^j,w_k^j)\}_{j=1}^n,\bar z_{k-1},q,\lam)$.
			\ENDFOR
		\end{algorithmic}\label{alg:PPM}
	\end{algorithm}
	
	Under Assumptions (A1)–(A5) on \eqref{eq:ProbIntro} and mild conditions on the input of Algorithm~\ref{alg:PPM}, 
	Theorem~\ref{thm:main} establishes a high-probability guarantee for obtaining an $\varepsilon$-solution of \eqref{eq:ProbIntro}. 
	In addition, Theorem~\ref{thm:sample} provides a bound on the sample complexity of stochastic gradients that is logarithmic in the confidence parameter. 
	A further contribution of this work is the observation that the oracle PSS (Algorithm~\ref{alg:PSS}) offers a new form of variance reduction without requiring mini-batches of samples (see the discussion at the end of Section~\ref{sec:PSS}). 
	This effect arises naturally from employing PPM within stochastic optimization, and highlights the intrinsic variance-reduction benefits of our approach.

	\section{Proximal Subproblem Solver}\label{sec:PSS}
	
	This section introduces and analyzes the first key oracle used in Algorithm~\ref{alg:PPM}, namely the proximal subproblem solver (PSS). 
	A detailed description of PSS is provided in Algorithm~\ref{alg:PSS}.
	
	Our goal in this section is to study the following proximal subproblem
	\begin{equation}\label{def:hatx}
		\hat x:=\underset{x\in \R^d} \argmin \left\{ \phi^{\lam}(x):=\phi(x) + \frac{1}{2\lam}\|x-x_0\|^2\right\}
	\end{equation}
	and analyze the solution pair $(x_{I+1},y_{I+1})$ returned by Algorithm \ref{alg:PSS}. The central result is Proposition~\ref{prop:rsk}, which establishes the convergence rate of the expected primal gap of \eqref{def:hatx}. 
	More importantly, it demonstrates that Algorithm \ref{alg:PSS} significantly reduces the variance term within the bound of the expected primal gap.
	
	\begin{algorithm}[H]
		\caption{Proximal Subproblem Solver, PSS$(x_0,\alpha,\lam,I)$}
		\begin{algorithmic}
			\STATE \textbf{Input:} Initial point $ x_0\in \dom h $, scalars $\alpha \in (0,1)$ and $ \lam>0$, and integer $I \ge 1$.
			\FOR{$i = 1, \ldots, I+1$}
			\STATE {\bf Step 1.} Take an independent sample $\xi_{i-1}$ of r.v.\ $\xi$ and set $s_{i-1}=s(x_{i-1},\xi_{i-1})$;
			\STATE {\bf Step 2.} Compute
			\lessgap
			\lessgap
			\begin{align}
				x_{i} =\underset{x\in \R^d}\argmin
				\left\lbrace  
				h(x) + \inner{S_i}{x} +\frac{1}{2\lam}\|x- x_0 \|^2 \right\rbrace,  \label{def:xj} 
			\end{align}
			\lessgap
			\lessgap
			\begin{align} \label{def:yj}
				y_i =  \left\{\begin{array}{ll}
					x_{i}, & \text { if } i=1 , \\ 
					\alpha y_{i-1} + (1-\alpha) x_i,  & \text { otherwise},
				\end{array}\right.
			\end{align}	
			\lessgap
			\lessgap
			where
			\begin{align} \label{eq:Sj}
				S_i =  \left\{\begin{array}{ll}
					s_0, & \text { if } i=1 , \\ 
					\alpha S_{i-1} + (1-\alpha)
					s_{i-1} ,  & \text { otherwise}.
				\end{array}\right.
			\end{align}	
			\lessgap
			\lessgap
			\ENDFOR
			\STATE \textbf{Output:} $x_{I+1}$ and $y_{I+1}$.
		\end{algorithmic}\label{alg:PSS}
	\end{algorithm}

	To streamline our presentation, we introduce the following definitions:
	\begin{equation}\label{def:Phi}
		\Phi(\cdot,\xi)=F(\cdot,\xi)+h(\cdot), \quad 
		\ell(\cdot;x,\xi)= f(x)+\inner{s(x,\xi)}{\cdot-x} + h(\cdot), 
	\end{equation}
	
	\begin{align} \label{def:u}
		u_i :=  \left\{\begin{array}{ll}
			\Phi(x_1,\xi_1) + \frac{1}{2\lam}\|x_1-x_0\|^2, & \text { if } i=1 , \\ 
			\alpha u_{i-1} + (1-\alpha) 
			\left[\phi(x_{i}) +\frac{1}{2\lam}\|x_i-x_0\|^2\right],  & \text { otherwise},
		\end{array}\right.
	\end{align}	
	
	\begin{align} \label{eq:Gamma}
		\cL_{i}(\cdot) :=  \left\{\begin{array}{ll}
			F(x_0,\xi_0) + 
			\langle s_0,\cdot-x_0\rangle + h(\cdot), & \text { if } i=1 , \\ 
			\alpha \cL_{i-1}(\cdot) + (1-\alpha)
			\ell(\cdot;x_{i-1},\xi_{i-1}),  & \text { otherwise},
		\end{array}\right.
	\end{align}	
	and
	\begin{equation}\label{def:func-lam}
		\cL_i^\lam(\cdot):=\cL_i(\cdot)+\frac{1}{2\lam}\|\cdot-x_0\|^2, \quad t_i:=u_i-\cL_i^\lam(x_i).
	\end{equation}

	We are ready to state the main result of this section. The following proposition demonstrates that the bound $\E[t_{I+1}]$ on the expected primal gap of \eqref{def:hatx} decreases as $I$ increases. 
	
	\begin{proposition}\label{prop:rsk}
		Assuming 
		\begin{equation}\label{def:tau}
			\alpha \ge \frac{I/2 + \lam L}{1+ I/2 + \lam L},
		\end{equation}
		then we have
		\begin{equation}\label{ineq:expectation}
			\E[t_{I+1}]\le \alpha^I \left(\sigma D + \frac{L D^2}2\right) + \frac{\lam \sigma^2}{I}.
		\end{equation}  
	\end{proposition}
	
	The full proof of Proposition \ref{prop:rsk} is deferred to the end of the section. Next, we develop several intermediate technical results that will serve as key building blocks for the proof of Proposition~\ref{prop:rsk}.
	We begin by observing some relations.
	It is easy to see from \eqref{def:xj}, \eqref{eq:Sj}, and \eqref{def:func-lam} that
	\begin{equation}\label{eq:xj}
		x_i =\underset{x\in \R^d}\argmin \cL_{i}^\lam(x).
	\end{equation}
	This observation and the fact that $\cL_i^\lam$ is $(1/\lam)$-strongly convex imply that for every $x\in \dom h$,
	\begin{equation}\label{ineq:xj}
		\cL_{i}^\lam(x) \ge \cL_{i}^\lam(x_i) + \frac{1}{2\lam}\|x- x_i \|^2.
	\end{equation}
	
	The result outlined below provides some basic relations that are frequently used in our analysis. 
	
	\begin{lemma}\label{lem:101}
		For every $i\ge 1$, we have
		\begin{align}
			\E[\phi^\lam( y_i)] &\le \E[u_i], \label{ineq:basic1} \\
			\E[\cL_i(x)]  &\le \phi(x), \quad \forall x \in \dom h, \label{ineq:basic2} \\
			\phi(x_i) - \ell(x_{i};x_{i-1},\xi_{i-1})
			&\le \|\nabla f(x_{i-1}) - s_{i-1}\| \|x_i-x_{i-1}\| + \frac{L}{2}\|x_i-x_{i-1}\|^2. \label{ineq:Phi}
		\end{align}
	\end{lemma}
	
	\begin{proof}
		We first prove \eqref{ineq:basic1} by induction. It is easy to verify that \eqref{ineq:basic1} holds for $i=1$ using \eqref{def:u} and assumption (A1).
		Assume that \eqref{ineq:basic1} holds for some $i\ge 1$. Then, using \eqref{def:yj}, \eqref{def:u}, the induction hypothesis, and the convexity of $\phi^\lam$, we conclude that
		\[
		\E[u_{i+1}]  \overset{\eqref{def:u},\eqref{ineq:basic1}}{\ge} \alpha \E[\phi^\lam(y_i)] + (1-\alpha) \E[\phi^\lam(x_{i+1})]
		\ge \E[\phi^\lam(\alpha y_i + (1-\alpha)x_{i+1})]  \overset{\eqref{def:yj}}{=} \E[\phi^\lam(y_{i+1})].
		\]
		Now, we prove \eqref{ineq:basic2} again by induction.
		It is easy to verify that \eqref{ineq:basic2} holds for $i=1$ using \eqref{eq:Gamma}, assumption (A1), and the convexity of $f$.
		Assume that \eqref{ineq:basic2} holds for some $i\ge 1$. Then, using \eqref{eq:Gamma}, the induction hypothesis, and the convexity of $\phi^\lam$, we conclude that
		\[
		\E[\cL_{i+1}(x)] \overset{\eqref{eq:Gamma}}{=} \alpha \E[\cL_i(x)] + (1-\alpha) \E[\ell(x;x_{i},\xi_{i})] 
		\overset{\eqref{ineq:basic2}}{\le} \alpha \phi(x) + (1-\alpha) \phi(x) = \phi(x).
		\]
		Finally, we prove \eqref{ineq:Phi}.
		Using the definitions of $\phi$ and $\ell(\cdot;x,\xi)$ in \eqref{eq:ProbIntro} and \eqref{def:Phi}, respectively, we have
		\begin{align*}
			\phi(x_i) - \ell(x_{i};x_{i-1},\xi_{i-1})
			&= f(x_{i}) - f(x_{i-1}) - \inner{s_{i-1}}{x_{i}-x_{i-1}} \\
			&\le \inner{\nabla f(x_{i-1}) - s_{i-1}}{x_i-x_{i-1}} + \frac{L}{2}\|x_i-x_{i-1}\|^2,
		\end{align*}
		where the inequality is due to the second inequality in \eqref{ineq:twoside}. Now, \eqref{ineq:Phi} directly follows from the above inequality and the Cauchy-Schwarz inequality.
	\end{proof}
	
	The next technical result shows that the expectation $\E[t_i]$ provides an upper bound on the primal gap of \eqref{def:hatx}. 
	
	\begin{lemma}\label{lem:rt}
		For every $i \ge 1$, define
		\begin{equation}\label{def:delta}
			r_{i}:= \frac{ \lam \|\nabla f(x_{i}) - s_{i}\|^2}{I}.
		\end{equation}
		Then, the following statements hold:
		\begin{itemize}
			\item[(a)] for every $i\ge 1$, $\E[r_i] \le \lam \sigma^2/I$;
			
			\item[(b)] $\E[t_1]\le \sigma D + L D^2/2$ where $\sigma$ and $D$ are as in Assumptions (A2) and (A5), respectively;
			
			\item[(c)] for every $i\ge 1$, $\E[t_i]\ge \E[\phi^\lam( y_i) - \phi^\lam(\hat x)]$.
		\end{itemize}
	\end{lemma}
	
	\begin{proof}
		(a) This statement follows directly from \eqref{def:delta}, the fact that $s_i=s(x_{i},\xi_{i})$, and assumption~(A2).
		
		(b) Let 
		\begin{equation}\label{def:error}
			e_i := \Phi(x_{i},\xi_{i})-\phi(x_{i})
			= F(x_{i},\xi_{i})-f(x_{i}).
		\end{equation}
		Using the definitions of $t_i$ and $\cL_i^\lam$ in \eqref{def:func-lam}, and $u_i$ in \eqref{def:u}, we have
		\begin{align}
			t_1 & \stackrel{\eqref{def:func-lam}}{=} u_1 - \cL_1^{\lam}(x_1) \nn \\ &\stackrel{\eqref{def:u},\eqref{def:func-lam}}{=} \Phi(x_1,\xi_1) - \left[F(x_0,\xi_0) + \inner{s_0}{x_1-x_0} + h(x_1)\right] \nn \\
			& \stackrel{\eqref{def:Phi},\eqref{def:error}}= e_1 - e_0 + \phi(x_1) - \ell(x_1;x_0,\xi_0) \nn \\
			&\stackrel{\eqref{ineq:Phi}}\le e_1 - e_0
			+ \|\nabla f(x_0) - s_0\| \|x_1-x_0\| + \frac{L}{2}\|x_1-x_0\|^2, \label{eq:tik}
		\end{align}
		where the inequality is due to \eqref{ineq:Phi}.
		Thus, the above inequality and assumption (A5) imply that
		\begin{equation}\label{ineq:t1-1}
			t_1 \le  e_1 - e_0
			+ \|\nabla f(x_0) - s_0\| D + \frac{L}{2} D^2.
		\end{equation}
		It follows from \eqref{def:error}, and assumptions (A1) and (A2) that 
		\[
		\E[e_1]=0, \quad \E[e_0]=0, \quad \E[\|\nabla f(x_0) - s_0\|^2] \le \sigma^2.
		\]
		Hence, the statement follows by taking expectation of \eqref{ineq:t1-1}
		and using the above three relations.
		
		(c) It follows from \eqref{eq:xj} and \eqref{ineq:basic2} that 
		\[
		\E[\cL_{i}^\lam(x_i)] \stackrel{\eqref{eq:xj}}\le \E[\cL_{i}^\lam(\hat x)] \stackrel{\eqref{ineq:basic2}}\le \E[\phi^\lam(\hat x)].
		\]
		Using the above inequality and \eqref{ineq:basic1}, we have
		\[
		\E[\phi^\lam( y_i) - \phi^\lam(\hat x)] \le \E[u_i-\cL_i^\lam(x_i)].
		\]
		Hence, the last statement follows from the definition of $t_i$ in \eqref{ineq:basic2}.
	\end{proof}

	The next lemma establishes a relation between $t_i$ and $r_i$, defined in \eqref{def:func-lam} and \eqref{def:delta}, respectively. 
	This relation plays a key role in the proof of Proposition~\ref{prop:rsk}, where it is used to show that $t_{I+1}$ is small in expectation. 
	
	\begin{lemma}\label{lem:102}
		Assuming \eqref{def:tau} holds,
		then for every $i\ge 2$, we have
		\begin{equation} \label{ineq:recursive}
			t_i  \le 
			\alpha^{i-1} t_1 + (1-\alpha)
			\sum_{j=1}^{i-1} \alpha^{i-j-1} r_j,
		\end{equation}
		where $t_i$ and $r_i$ are as in \eqref{def:func-lam} and \eqref{def:delta}, respectively.
	\end{lemma}
	
	\begin{proof} 
		It suffices to prove that for every $i \geq 2$,
		\begin{equation}\label{ineq:tj}
			t_{i} \le \alpha t_{i-1} + (1-\alpha) r_{i-1},
		\end{equation}
		since it is clear that \eqref{ineq:recursive} follows immediately from \eqref{ineq:tj}
		and an induction argument.    
		Let $i \geq 2$ be given.
		It follows from the definitions of $\cL_{i}$ and $\cL_{i}^\lam$ in \eqref{eq:Gamma} and \eqref{def:func-lam}, respectively, that
		\begin{align*}
			\cL_{i}^\lam(x_{i}) - (1-\alpha) \ell(x_{i};x_{i-1},\xi_{i-1})& =  \alpha \cL_{i-1}(x_{i})+\frac{1}{2\lambda}\|x_i-x_0\|^2 \\
			& = \alpha \cL_{i-1}^\lam(x_{i}) + \frac{1-\alpha}{2\lambda}\|x_i-x_0\|^2 \\
			&\stackrel{\eqref{ineq:xj}}\ge \alpha \left[\cL_{i-1}^\lam(x_{i-1}) + \frac{1}{2\lam}\|x_{i}-x_{i-1}\|^2\right] + \frac{1-\alpha}{2\lambda}\|x_i-x_0\|^2,
		\end{align*}
		where the inequality is due to \eqref{ineq:xj}.
		Rearranging the terms in the above inequality and using \eqref{def:func-lam}, \eqref{ineq:Phi}, \eqref{def:delta}, and \eqref{def:tau}, we have
		\begin{align*}
			& \cL_{i}^\lam(x_{i})  - \alpha \cL_{i-1}^\lam(x_{i-1}) \ge  (1-\alpha) \left[ 
			\ell(x_{i};x_{i-1},\xi_{i-1}) + \frac{1}{2\lambda}\|x_i-x_0\|^2 + \frac{\alpha}{2\lam(1-\alpha)} \|x_i-x_{i-1}\|^2 \right] \\
			&\overset{\eqref{def:func-lam},\eqref{ineq:Phi}}{\ge} (1-\alpha)\phi^\lam(x_{i}) + (1-\alpha) \left[\frac{\alpha}{2\lam(1-\alpha)} \|x_{i}-x_{i-1}\|^2 - \|\nabla f(x_{i-1}) - s_{i-1}\| \|x_i-x_{i-1}\| - \frac{L}{2}\|x_i-x_{i-1}\|^2 \right] \\
			&\stackrel{\eqref{def:tau}}\ge (1-\alpha)\phi^\lam(x_{i}) + (1-\alpha) \left[\frac{I}{4\lam} \|x_{i}-x_{i-1}\|^2 - \|\nabla f(x_{i-1}) - s_{i-1}\| \|x_i-x_{i-1}\| \right] \\
			&\ge  (1-\alpha) \phi^\lam(x_i) - (1-\alpha) \frac{ \lam\|\nabla f(x_{i-1}) - s_{i-1}\|^2}{I} \stackrel{\eqref{def:delta}}= (1-\alpha) \phi^\lam(x_i) - (1-\alpha) r_{i-1},
		\end{align*}
		where the last inequality is by the AM-GM inequality.
		Rearranging the above
		inequality and using the definition of $t_i$ in
		\eqref{def:func-lam}, identity
		\eqref{def:u}, and the fact that $i\ge 2$, we then conclude that
		$$
		\begin{array}{lcl}
			\cL_{i}^\lam(x_{i}) + (1-\alpha)r_{i-1} &\ge &
			(1-\alpha) \phi^\lam(x_i) + \alpha \cL_{i-1}^\lam(x_{i-1}) \\
			& \overset{\eqref{def:func-lam}}{=} &
			(1-\alpha) \phi^\lam(x_i) + \alpha (u_{i-1} - t_{i-1})
			\overset{\eqref{def:u}}{=} u_i - \alpha t_{i-1},
		\end{array}
		$$
		which, in view of the definition of
		$t_i$ in \eqref{def:func-lam} again, implies \eqref{ineq:tj}. 
	\end{proof}
	
	\vgap
	
	We are now ready to prove Proposition \ref{prop:rsk}.
	
	\vgap
	
	\noindent
	{\bf Proof of Proposition \ref{prop:rsk}: } 
	It follows from Lemma \ref{lem:102} that
	\[
	t_{I+1} \le \alpha^{I}t_1 + (1-\alpha) \sum_{j=1}^{I} \alpha^{I-j} r_j.
	\]
	Taking the expectation of the above inequality and using Lemma \ref{lem:rt},  we have
	\[
	\E[t_{I+1}] 
	\le \alpha^I \E[t_1] + (1-\alpha)
	\sum_{j=1}^{I} \alpha^{I-j} \E[r_j]
	\le \alpha^I \left(\sigma D + \frac{L D^2}2\right)  + (1-\alpha) \frac{\lam \sigma^2}{I}  \sum_{j=1}^{I} \alpha^{I-j}.
	\]
	Therefore, \eqref{ineq:expectation} immediately holds.
	\QEDA
	
	\vgap
	
	We conclude this section by offering some insight into Proposition~\ref{prop:rsk}. 
	First, in \eqref{ineq:expectation}, the terms $\lambda \sigma^2/I$ and $\alpha^I \left(\sigma D + L D^2/2\right)$ can be interpreted as variance and bias components, respectively. 
	Second, the tradeoff between these components is governed by the choice of the prox stepsize $\lambda$, since the bias term depends on $\lambda$ through $\alpha$ as defined in \eqref{def:tau}. 
	Third, increasing the number of iterations $I$ in Algorithm~\ref{alg:PSS} reduces both the bias and variance terms. 
	Unlike standard stochastic gradient methods, inequality~\eqref{ineq:expectation} indicates a variance reduction by a factor of $I$, not through sample averaging but by performing multiple iterations to solve the proximal subproblem \eqref{def:hatx}. 
	This distinctive variance-reduction effect is natural, as Algorithm~\ref{alg:PSS} effectively uses $I+1$ stochastic gradient samples.

	\section{Analysis of Proximal Point Method}\label{sec:PPM}
	
	This section analyzes Step 1 of Algorithm \ref{alg:PPM} and prepares the technical results necessary for the analysis of the PB oracle in Step 2 of Algorithm \ref{alg:PPM}. The main result in this section is Proposition~\ref{prop:step1}, which gives a probability guarantee of the suboptimality of the proximal subproblem \eqref{eq:prox}.
	
	Before starting the analysis, we need to properly translate the notation from an inner viewpoint to an outer (proximal point) perspective, i.e., using SPPM (Algorithm~\ref{alg:PPM}) iteration index $k$ instead of PSS (Algorithm~\ref{alg:PSS}) iteration index $i$.
	Consider the $j$-th call to PSS in Step 1 of Algorithm \ref{alg:PPM}, and let $x_0^j$ and $(x_{I+1}^j,y_{I+1}^j)$ denote the initial point and output for PSS, respectively.
	We know that for every $j\in \{1,\ldots, n\}$,
	\begin{equation}\label{eq:z-}
		\bar z_{k-1}=x_0^j, \quad z_k^j= x_{I+1}^j, \quad w_k^j=y_{I+1}^j.
	\end{equation}
	For simplicity, we denote $z_k^j$ and $w_k^j$ by $z_k$ and $w_k$, respectively, ignoring the query index $j$. Therefore, in view of \eqref{eq:z-}, the notational convention we adopt in this section is
	\begin{equation}\label{eq:equiv}
		\bar z_{k-1}=x_0^j, \quad z_k = x_{I+1}^j, \quad w_k = y_{I+1}^j.
	\end{equation}
	Although we omit the index $j$ for simplicity, we note that all the results in this section hold for each $j\in \{1,\ldots, n\}$, i.e., any call to PSS in Step 1 of Algorithm \ref{alg:PPM}.

	We are now ready to state the main result of Section \ref{sec:PPM}, which provides a probability guarantee of the suboptimality of the proximal subproblem \eqref{eq:prox}. 
	
	\begin{proposition}\label{prop:step1}
		For every $k\ge 1$, 
		we have
		\begin{equation}\label{ineq:low-prob}
			\Prob \left( \phi(w_k)+\frac1{2\lam}\|w_k-\bar z_{k-1}\|^2 -  \phi(\hat z_k) - \frac1{2\lam} \|\hat z_k-\bar z_{k-1}\|^2 + \frac{1+\lam \mu}{\lam(2+\lam \mu)}\|\hat z_k-z_k\|^2 \le 8\varepsilon_k \right) \ge \frac34,
		\end{equation}
		where $\hat z_k$ is as in \eqref{eq:prox}.
	\end{proposition}
	
	Proposition~\ref{prop:step1} follows directly from Markov’s inequality together with Lemma~\ref{lem:exp}, which we establish at the end of this section. To this end, we first develop several intermediate technical results that serve as key ingredients towards the proof of Lemma~\ref{lem:exp}.
	
	We begin by using the notation in \eqref{eq:equiv} to restate key results from Section~\ref{sec:PSS} in the proximal point framework. 
	The following lemma serves as the starting point for our proximal point analysis.

	\begin{lemma}\label{lem:translate}
		For every $k\ge 1$, let
		\begin{equation}\label{def:Gammak}
			\Gamma_k = \cL_{I+1}^{j}, \quad \varepsilon_k = \alpha^I \left(\sigma D + \frac{L D^2}2\right) + \frac{\lam \sigma^2}{I},
		\end{equation}
		where $\cL_{I+1}^{j}$ means $\cL_{I+1}$ considered in the analysis of the $j$-th call to PSS in Step 1 of Algorithm \ref{alg:PPM} and $j$ is an arbitrary index in $\{1,\ldots,n\}$. Then, the following relations hold
		\begin{align}
			&\E[ \Gamma_k(x)] \le \phi(x), \quad \forall x \in \dom h, \label{ineq:under} \\
			&z_k = \underset{x\in \R^d} \argmin \left\{\Gamma_k (x) + \frac1{2\lam} \|x-\bar z_{k-1}\|^2 \right\}, \label{def:x} \\
			&\E\left[\phi(w_k) + \frac1{2\lam}\|w_k-\bar z_{k-1}\|^2  - \Gamma_k (z_k) - \frac1{2\lam} \|z_k-\bar z_{k-1}\|^2\right] \le \varepsilon_k.  \label{ineq:eps}
		\end{align}
	\end{lemma}
	
	\begin{proof}
		Consider the $j$-th call to PSS (i.e., Algorithm \ref{alg:PSS}).
		It clearly follows from \eqref{ineq:basic2} and the definition of $\Gamma_k$ in \eqref{def:Gammak} that \eqref{ineq:under} holds.
		Now, we prove \eqref{def:x}.
		It follows from \eqref{eq:xj} with $i=I+1$ that
		\[
		x_{I+1}^j =\underset{x\in \R^d}\argmin
		\left\lbrace  
		\cL_{I+1}^j(x) +\frac{1}{2\lam}\|x- x_0^j\|^2 \right\rbrace.  
		\]
		In view of \eqref{eq:equiv} and the definition of $\Gamma_k$ in \eqref{def:Gammak}, the above identity immediately implies \eqref{def:x}.
		Finally, we prove \eqref{ineq:eps}.
		Using \eqref{ineq:expectation} and the definitions of $t_i$ and $\varepsilon_k$ in \eqref{def:func-lam} and \eqref{def:Gammak}, resoectively, we have 
		\[
		\varepsilon_k \stackrel{\eqref{ineq:expectation},\eqref{def:Gammak}}\ge \E\left[t_{I+1}^j\right] \stackrel{\eqref{def:func-lam}}= \E\left[u_{I+1}^j-(\cL_{I+1}^j)^{\lam}(x_{I+1}^j)\right].
		\]
		It thus follows from \eqref{ineq:basic1} and the definitions of $\phi^\lam$ and $\cL_i^\lam$ in \eqref{def:hatx} and \eqref{def:func-lam}, respectively, that
		\begin{align*}
			\varepsilon_k & \stackrel{\eqref{ineq:basic1}}\ge \E\left[\phi^\lam( y_{I+1}^j) - (\cL_{I+1}^j)^{\lam}(x_{I+1}^j)\right] \\
			& \stackrel{\eqref{def:hatx},\eqref{def:func-lam}}= \E\left[\phi( y_{I+1}^j) + \frac{1}{2\lam}\|y_{I+1}^j-x_0^j\|^2 - \cL_{I+1}^j(x_{I+1}^j) - \frac{1}{2\lam}\|x_{I+1}^j-x_0^j\|^2\right].
		\end{align*}
		In view of $(\bar z_{k-1},z_k,w_k)$ and $\Gamma_k$ given in \eqref{eq:equiv} and \eqref{def:Gammak}, respectively, the above inequality immediately implies \eqref{ineq:eps}.
	\end{proof}

	The following lemma is a technical result that translates the stationarity condition of~\eqref{def:x} into conditions resembling the convexity inequality.
	
	\begin{lemma} \label{lem:error}
		Define
		\begin{equation}\label{def:s}
			v_k := \frac{\bar z_{k-1}-z_k}\lam, \quad \eta_k := \phi(w_k) - \Gamma_k(z_k) - \inner{v_k}{w_k-z_k}.
		\end{equation}
		Then, we have
		\begin{align}
			\phi(x) &\ge \E[\phi(w_k) + \inner{v_k}{x-w_k} - \eta_k], \quad \forall x\in \dom h, \label{ineq:linear} \\
			\E[\eta_k] &\le \varepsilon_k  - \frac1{2\lam} \E\left[\|w_k-z_k \|^2 \right]. \label{ineq:eta}
		\end{align}
	\end{lemma}
	
	\begin{proof}
		It follows from the definition of $v_k$ in \eqref{def:s} and the optimality condition of \eqref{def:x} that $v_k  \in \partial \Gamma_k(z_k)$, i.e., for every $ x\in \dom h$,
		\[
		\Gamma_k(x) \ge \Gamma_k(z_k) + \inner{v_k}{x-z_k} \stackrel{\eqref{def:s}}= \phi(w_k) + \inner{v_k}{x-w_k} - \eta_k,
		\]
		where the identity is due to the definition of $\eta_k$ in \eqref{def:s}.
		Thus, \eqref{ineq:linear} follows by taking expectation of the above inequality and using \eqref{ineq:under}.
		Now, we prove \eqref{ineq:eta}.
		Using \eqref{ineq:eps} and the definition of $\eta_k$ in \eqref{def:s}, we have
		\begin{align*}
			\E[\eta_k] &= \E[\phi(w_k) - \Gamma_k(z_k) - \inner{v_k}{w_k-z_k}] \\
			&\le \varepsilon_k + \E\left[- \frac1{2\lam}\|w_k-\bar z_{k-1}\|^2  +  \frac1{2\lam}\|z_k-\bar z_{k-1}\|^2 - \inner{v_k}{w_k-z_k}\right] \\
			&= 
			\varepsilon_k  - \E\left[\frac1{2\lam}
			\|w_k-z_k \|^2 \right],
		\end{align*}
		where the last identity is due to the definition of $v_k$ in \eqref{def:s}.
	\end{proof}

	The lemma presented next is crucial for re-establishing the $\mu$-strong convexity of $\phi$. To achieve this, we introduce an auxiliary function $q_k$ and outline its key properties in the subsequent discussion.
	
	\begin{lemma}\label{lem:qk}
		Define
		\begin{equation}\label{def:Gamma}
			q_k(x):=\phi(w_k) + \inner{v_k}{x-w_k} + \frac{\mu}4 \|x-w_k\|^2 - 2 \eta_k.
		\end{equation}
		Then, the following statements hold:
		\begin{itemize}
			\item[a)] $\E[q_k(x)]\le \phi(x)$ for every $x \in \dom h$;
			\item[b)] $\E[q_k(z_k)] \ge \E[\Gamma_k(z_k)]  + \left(\frac{1}{2\lam} + \frac{\mu}4\right) \E[\|z_k-w_k\|^2] - \varepsilon_k$.
		\end{itemize}
	\end{lemma}
	
	\begin{proof}
		a) Assumption (A5) implies that $\phi(x)-\E[\inner{v_k}{x}]$ has a unique global minimum $\bar y $. Thus, for every $x \in \dom h$, we have
		\begin{equation}\label{ineq:u}
			\phi(x)\ge \phi(\bar y)+\E[\inner{v_k}{x-\bar y}]+\frac{\mu}{2}\|x-\bar y\|^2.
		\end{equation}
		It follows from \eqref{ineq:linear} with $x=\bar y$ that 
		\begin{equation}\label{ineq:bar y}
			\phi(\bar y)\ge \E[\phi(w_k) + \inner{v_k}{\bar y-w_k} - \eta_k]
		\end{equation}
		Combining \eqref{ineq:u} and \eqref{ineq:bar y}, we conclude that for every $x \in \dom h$,
		\begin{align}
			\phi(x)&
			\ge \E[\phi(w_k)+\inner{v_k}{x- w_k}-\eta_k]+\frac{\mu}{2}\|x-\bar y\|^2 \nn \\
			&= \E[\phi(w_k)+\inner{v_k}{x- w_k}]-\E\left[\eta_k+\frac{\mu}{2}\|\bar y-w_k\|^2 \right] +\frac{\mu}{2}(\E[\|\bar y-w_k\|^2]+\|x-\bar y\|^2) \nn \\
			&\ge \E[\phi(w_k)+\inner{v_k}{x- w_k}-\eta_k']+\frac{\mu}{4}\E[\|x-w_k\|^2] \label{ineq:u y}
		\end{align}
		where
		\begin{equation}\label{eq:etap}
			\eta_k':=\eta_k+\frac\mu2 \|\bar y-w_k\|^2.
		\end{equation}
		Using \eqref{ineq:u} with $x=w_k$ and taking expectation of the resulting inequality, we have
		\[
		\E[\phi(w_k)]\ge \phi(\bar y)+\E[\inner{v_k}{w_k-\bar y}]+\frac{\mu}{2}\E[\|w_k-\bar y\|^2].
		\]
		It follows from the above inequality and \eqref{ineq:bar y} that
		\[
		\frac{\mu}{2}\E[\|w_k-\bar y\|^2] \le \E[\eta_k],
		\]
		which together with \eqref{eq:etap} implies that
		\[
		\E[\eta_k'] \le 2 \E[\eta_k].
		\]
		Statement (a) now follows from \eqref{ineq:u y}, the above inequality, and the definition of $q_k$ in \eqref{def:Gamma}.
		
		b) Taking $x=z_k$ in \eqref{def:Gamma} and using the definition of $\eta_k$ in \eqref{def:s}, we have
		\[
		q_k(z_k) = \phi(w_k) + \inner{v_k}{z_k-w_k} + \frac{\mu}4 \|z_k-w_k\|^2 - 2 \eta_k = \Gamma_k(z_k) + \frac{\mu}4 \|z_k-w_k\|^2 - \eta_k.
		\]
		The statement now follows from taking expectation of the above inequality and using \eqref{ineq:eta}.
	\end{proof}

	Combining the technical results above, the next lemma shows that the suboptimality of \eqref{eq:prox} is under control in expectation. 
	
	\begin{lemma}\label{lem:exp}
		For every $k\ge 1$ and $x\in \dom h$, we have
		\[
		\E \left[ \phi(w_k)+\frac1{2\lam}\|w_k-\bar z_{k-1}\|^2 -  \phi(x) - \frac1{2\lam} \|x-\bar z_{k-1}\|^2 + \frac{1+\lam \mu}{\lam(2+\lam \mu)}\|x-z_k\|^2 \right] \le 2 \varepsilon_k.
		\]
	\end{lemma}
	
	\begin{proof}
		Using the definition of $q_k$ in \eqref{def:Gamma}, it can be verified that
		\begin{align*}
			&\E[q_k(x)] + \frac1{2\lam} \E[\|x-\bar z_{k-1}\|^2] - \left(\frac1{2\lam}+\frac{\mu}{4}\right) \E[\|x-z_k\|^2]\\
			=&\E[q_k(z_k)] + \frac1{2\lam} \E[\|z_k-\bar z_{k-1}\|^2] + \frac\mu2 \E[\inner{z_k-w_k}{x-z_k}].
		\end{align*}
		It thus follows from Lemma \ref{lem:qk}(b) that
		\begin{align}
			&\E[q_k(x)] + \frac1{2\lam} \E[\|x-\bar z_{k-1}\|^2] -\left(\frac1{2\lam}+\frac{\mu}{4}\right) \E[\|x-z_k\|^2] \nn \\
			\ge & \E[\Gamma_k(z_k)] + \frac1{2\lam} \E[\|z_k-\bar z_{k-1}\|^2] - \varepsilon_k + \left(\frac{1}{2\lam} + \frac{\mu}4\right) \E[\|z_k-w_k\|^2] + \frac\mu2 \E[\inner{z_k-w_k}{x-z_k}]. \label{ineq:qk}
		\end{align}
		Note that by the AM-GM inequality and the Cauchy-Schwarz inequality, we have
		\[
		\left(\frac{1}{2\lam} + \frac{\mu}4\right) \|z_k-w_k\|^2 + \frac{\lam \mu^2}{4(2+\lam \mu)}\|x-z_k\|^2 \ge -\frac\mu2 \inner{z_k-w_k}{x-z_k},
		\]
		and hence
		\[
		\left(\frac{1}{2\lam} + \frac{\mu}4\right) \E[\|z_k-w_k\|^2] + \frac\mu2 \E[\inner{z_k-w_k}{x-z_k}] \ge - \frac{\lam \mu^2}{4(2+\lam \mu)}\E[\|x-z_k\|^2].
		\]
		Plugging the above inequality into \eqref{ineq:qk} and using \eqref{ineq:eps}, we obtain
		\begin{align*}
			&\E[q_k(x)] + \frac1{2\lam} \E[\|x-\bar z_{k-1}\|^2] -\left(\frac1{2\lam}+\frac{\mu}{4}\right) \E[\|x-z_k\|^2] \\
			\ge & \E[\phi(w_k)] + \frac1{2\lam} \E[\|w_k-\bar z_{k-1}\|^2] - 2\varepsilon_k - \frac{\lam \mu^2}{4(2+\lam \mu)}\E[\|x-z_k\|^2].
		\end{align*}
		The lemma now follows from the above inequality, Lemma \ref{lem:qk}(a), and rearranging the terms.
	\end{proof}

	\section{Probability Booster}\label{sec:PB}
	
	This section introduces and analyzes the second key oracle used in Algorithm~\ref{alg:PPM}, namely PB. 
	It is designed to select, with high probability, a sufficiently accurate candidate from several approximate proximal solutions, thus ensuring the reliability of the overall scheme. 
	A detailed description is provided in Algorithm~\ref{alg:PB}. 
	PB itself relies on two auxiliary subroutines: second tertile selection (STS) and robust gradient estimation (RGE), whose presentations are deferred to Appendices~\ref{sec:sts} and~\ref{sec:rge} for clarity. 
	The design of Algorithm~\ref{alg:PB} is motivated by Algorithm~9 of \cite{davis2021low}.

	\begin{algorithm}[H]
		\caption{Probability Booster, PB$(\{(z^j,w^j)\}_{j=1}^n,\bar z, q,\lam)$}
		\begin{algorithmic}
			\STATE \textbf{Input:} Independent pairs $(z^1,w^1), \ldots, (z^n,w^n)$ generated by PSS with initial point $\bar z$, integer $q\ge 1$, and scalar $\lam>0$.
			\STATE {\bf Step 1.} Call oracle $\mathcal{J}_1=\operatorname{STS}(\{w^j\}_{j=1}^n, d_2(\cdot,\cdot))$;
			\STATE {\bf Step 2.} Call oracle $\mathcal{J}_2=\operatorname{STS}(\{z^j\}_{j=1}^n, d_2(\cdot,\cdot) )$;
			\STATE {\bf Step 3.} Fix an arbitrary $j_0 \in \mathcal{J}_1 \cap {\cal J}_2$ and set $\tilde w:=w^{j_0}$. Call oracle $\bar s(\tilde w)=\operatorname{RGE}(\tilde w,n,q)$;
			\STATE {\bf Step 4.} Define the metric $d_h(x, y):=|h(x)-h(y)+\langle\bar s(\tilde w) + (\tilde w-\bar z)/\lam, x-y\rangle|$. Call oracle $\mathcal{J}_3=\operatorname{STS}(\{w^j\}_{j=1}^n, d_h(\cdot,\cdot))$.
			\STATE \textbf{Output:} A pair $(z^j,w^j)$ for an arbitrary $j \in \mathcal{J}_1 \cap \mathcal{J}_2 \cap \mathcal{J}_3$.
		\end{algorithmic}\label{alg:PB}
	\end{algorithm}
	
	Note that the subroutine STS, used in Steps~1, 2, and~4, requires a metric as input. 
	Here, $d_2(x,y)=\|x-y\|_2$ denotes the Euclidean distance and is therefore a valid metric, while Lemma~\ref{lem:dh} establishes that $d_h(\cdot,\cdot)$ is also a metric. 
	
	Recall that Step~1 of Algorithm~\ref{alg:PPM} generates $n$ pairs $\{(z_k^j,w_k^j)\}_{j=1}^n$ by invoking the oracle PSS (Algorithm~\ref{alg:PSS}). 
	Proposition~\ref{prop:step1} provides a low-probability guarantee on the suboptimality of each pair $(z_k^j,w_k^j)$, as stated in \eqref{ineq:low-prob}. 
	The purpose of Algorithm~\ref{alg:PB} is to boost this low-probability bound into a high-probability guarantee by selecting one of the $n$ pairs such that an inequality analogous to \eqref{ineq:low-prob} holds with probability close to one.

	To simplify our notation, we exclude the iteration index $k$ and focus on the following proximal subproblem.
	Given the prox-center $\bar z \in \dom h$ and the prox stepsize $\lam>0$, define
	\begin{equation}\label{def:flam}
		\phi^\lam(\cdot):=\phi(\cdot)+\frac{1}{2\lam}\|\cdot-\bar z\|^2, \quad f^\lam(\cdot):=f(\cdot)+\frac{1}{2\lam}\|\cdot-\bar z\|^2,
	\end{equation}
	and consider
	\begin{equation}\label{eq:sub}
		\hat z:=\underset{x\in \R^d} \argmin \phi^{\lam}(x).
	\end{equation}
	The optimality condition of the proximal subproblem \eqref{eq:sub} gives 
	\[
	-\nabla f^\lam(\hat z) = -\left(\nabla f(\hat z) + \frac{\hat z - \bar z}{\lam} \right) \in \partial h(\hat z),
	\]
	where $\partial h(\hat z)$ denotes the subdifferential set of $h$ at $\hat z$. This implies that for every $x\in \dom h$,
	\begin{equation}\label{def:Bregman}
		D_h(x,\hat z) := h(x) - h(\hat z) + \inner{\nabla f^\lam(\hat z)}{x-\hat z} \ge 0,
	\end{equation}
	where $D_h$ is the Bregman divergence for $h$.
	Let $\bar s^\lam(\tilde w):=\bar s(\tilde w) + (\tilde w - \bar z)/\lam$. Then, it follows from the definition of $d_h(\cdot,\cdot)$ in Step 4 of Algorithm~\ref{alg:PB} that
	\begin{equation}\label{eq:dh}
		d_h(w^j,\hat z)=|h(w^j) - h(\hat z) + \inner{\bar s^\lam(\tilde w)}{w^j - \hat z}|.
	\end{equation}
	For every $j=1, \ldots, n$, define the event $B_j$ as follows
	\begin{equation}\label{eq:Aj}
		B_j:=\left\{ \phi^\lam(w^j) -  \phi^\lam(\hat z) + \frac{1+\lam \mu}{\lam(2+\lam \mu)} \|\hat z-z^j\|^2 \le \tau\right\},
	\end{equation}
	where $\tau$ is considered as (in view of \eqref{ineq:low-prob})
	\begin{equation}\label{eq:tau}
		\tau = 8 \varepsilon_k \stackrel{\eqref{def:Gammak}}= 8 \left[\alpha^I \left(\sigma D + \frac{L D^2}2\right) + \frac{\lam \sigma^2}{I}\right].
	\end{equation}
	Then it follows from Proposition \ref{prop:step1} that a low probability guarantee holds for each $B_j$, that is,
	\begin{equation}\label{ineq:prob}
		\Prob(B_j \text{ occurs}) \ge \frac34.
	\end{equation}
	
	The following result establishes that the output of Algorithm~\ref{alg:PB} satisfies a high-probability version of the inequality in \eqref{eq:Aj}, up to a factor of the condition number.

	\begin{proposition}\label{prop:boost}
		Assume that \eqref{def:tau} holds. If the input $q$ of Algorithm~\ref{alg:PB} satisfies
		\begin{equation}\label{ineq:q}
			q\ge \frac{18(1+\lam \mu)\sigma^2}{\lam L^2 \tau},
		\end{equation}
		where $\tau$ is as in \eqref{eq:tau}, then with probability at least $1-2 \exp\left( -n/72\right)$, the pair $(z^j,w^j)$ returned by PB satisfies
		\begin{equation}\label{ineq:phi}
			\phi^\lam(w^j) - \phi^\lam(\hat z) + \frac{1+\lam \mu}{\lam(2+\lam \mu)} \|\hat z-z^j\|^2 \le 12 \tau + 57 \kappa \tau,
		\end{equation}
		where $\hat z$ is as in \eqref{eq:sub} and condition number $\kappa=L/\mu>1$.
	\end{proposition}

	\begin{proof}
		Consider events $\{B_j\}_{j=1}^n$ defined in \eqref{eq:Aj} and define index set ${\cal J}_0$ and event $E_1$ as follows
		\[
		{\cal J}_0:=\left\{j\in\{1,\ldots,n\}: B_j \text{ occurs}\right\}, \quad E_1:= \left\{ |{\cal J}_0|> \frac {2n}3 \right\}.
		\]
		It follows from \eqref{ineq:prob} and Lemma \ref{lem:Hoeffding} that
		\begin{equation}\label{ineq:E1}
			\Prob\left(E_1 \right)\ge 1-\exp\left( -\frac{n}{72}\right).
		\end{equation}
		It clearly follows from the first inequality in \eqref{ineq:twoside1} that for every $j=1,\ldots,n$,
		\[
		\frac{1+\lam \mu}{2\lam}\|w^j-\hat z\|^2 + D_h(w^j,\hat z) + \frac{1+\lam \mu}{\lam(2+\lam \mu)} \|z^j-\hat z\|^2
		\le  \phi^\lam(w^j) - \phi^\lam(\hat z) + \frac{1+\lam \mu}{\lam(2+\lam \mu)} \|z^j-\hat z\|^2.
		\]
		The above inequality, \eqref{def:Bregman}, and \eqref{eq:Aj} imply that
		\begin{equation}\label{ineq:metric}
			\|w^j-\hat z\| \le \sqrt{\frac{2\lam\tau}{1+\lam \mu}}, \quad D_h(w^j,\hat z) \le \tau, \quad \|z^j-\hat z\| \le \sqrt{\frac{(2+\lam \mu) \lam\tau}{1+\lam \mu}}, \quad \forall j\in {\cal J}_0.
		\end{equation}
		
		\paragraph{Now, we are ready to analyze Steps 1 and 2 of Algorithm \ref{alg:PB}.}
		Assuming that the event $E_1$ occurs, then the three inequalities in \eqref{ineq:metric} hold for more than $2/3$ indices $j\in \{1,\dots,n\}$.
		Now, condition \eqref{cond:card} of Lemma \ref{lem:3eps} is satisfied. Applying Lemma \ref{lem:3eps} to Step 1 of Algorithm \ref{alg:PB} and using the first inequality in \eqref{ineq:metric}, we have
		\begin{equation}\label{ineq:J1}
			\|w^j-\hat z\| \stackrel{\eqref{ineq:metric},\eqref{ineq:3eps}}\le 3\sqrt{\frac{2\lam\tau}{1+\lam \mu}}, \quad \forall j \in {\cal J}_1.
		\end{equation}
		Applying Lemma \ref{lem:3eps} to Step 2 of Algorithm \ref{alg:PB} and using the third inequality in \eqref{ineq:metric}, we also have
		\begin{equation}\label{ineq:J2}
			\|z^j-\hat z\| \stackrel{\eqref{ineq:metric},\eqref{ineq:3eps}}\le 3\sqrt{\frac{(2+\lam \mu) \lam\tau}{1+\lam \mu}}, \quad \forall j \in {\cal J}_2.
		\end{equation}
		
		\paragraph{Next, we analyze Step 3 of Algorithm \ref{alg:PB}.}
		Condition \eqref{ineq:q} satisfies the assumption $q\ge 4\sigma^2/\delta^2$ in Lemma \ref{lem:RGE} with
		\[
		\delta=\frac L3 \sqrt{\frac{2\lam\tau}{1+\lam \mu}}.
		\]
		Hence, applying Lemma \ref{lem:RGE} to Step 3 of Algorithm \ref{alg:PB} and noting that $\nabla f(\tilde w) = \E[\bar s(\tilde w)]$, we have
		\begin{equation}\label{ineq:E21}
			\Prob\left(E_2 \,\middle\vert\, E_1 \right) \stackrel{\eqref{ineq:RGE}}\ge 1-\exp\left( -\frac{n}{72}\right),
		\end{equation}
		where event $E_2$ is defined as
		\begin{equation}\label{def:E2}
			E_2 := \left\{ \|\bar s(\tilde w) - \nabla f(\tilde w)\| \le L\sqrt{\frac{2\lam\tau}{1+\lam \mu}} \right\}.
		\end{equation}
		From now on, we suppose that $E_1\cap E_2$ occurs. Since $\tilde w=w^{j_0}$ with $j_0\in {\cal J}_1$, using the triangle inequality, \eqref{ineq:J1} with $j=j_0$, \eqref{def:E2}, and the fact that $f$ is $L$-smooth, we have 
		\begin{equation}\label{ineq:tw}
			\|\tilde w - \hat z\| \stackrel{\eqref{ineq:J1}}\le 3\sqrt{\frac{2\lam\tau}{1+\lam \mu}},
		\end{equation}
		and
		\begin{align}
			\|\bar s(\tilde w) - \nabla f(\hat z)\| 
			&\le \|\bar s(\tilde w) - \nabla f(\tilde w)\| + \|\nabla f(\tilde w) - \nabla f(\hat z)\| \nn \\
			&\stackrel{\eqref{def:E2}}\le L\sqrt{\frac{2\lam\tau}{1+\lam \mu}} + L \|\tilde w - \hat z\| 
			\stackrel{\eqref{ineq:tw}}\le 4 L \sqrt{\frac{2\lam\tau}{1+\lam \mu}}.\label{ineq:grad}
		\end{align}
		Using Lemma \ref{lem:dD}, \eqref{ineq:tw}, and \eqref{ineq:grad}, we have
		\[
		d_h(w^j,\hat z) 
		\stackrel{\eqref{ineq:d-D}}\le D_h(w^j,\hat z) + 4\left(L  + \frac{1}{\lam} \right) \sqrt{\frac{2\lam\tau}{1+\lam \mu}} \|w^j - \hat z\|, \quad \forall j\in\{1,\ldots,n\}.
		\]
		Hence, it follows from \eqref{ineq:metric} that
		\begin{equation}\label{ineq:dh}
			d_h(w^j,\hat z)
			\le \tau+ \frac{8 (1+\lam L) \tau}{1+\lam \mu}, \quad \forall j \in {\cal J}_0.
		\end{equation}
		
		\paragraph{Now, we are ready to analyze Step 4 of Algorithm~\ref{alg:PB}.}
		Note that \eqref{ineq:dh} holds for more than $2/3$ indices $j\in \{1,\dots,n\}$, which means condition~\eqref{cond:card} of Lemma~\ref{lem:3eps} is satisfied. Applying Lemma~\ref{lem:3eps} to Step 4 of Algorithm \ref{alg:PB} and using \eqref{ineq:dh}, we have
		\begin{equation}\label{ineq:J3}
			d_h(w^j,\hat z) \stackrel{\eqref{ineq:dh},\eqref{ineq:3eps}}\le 3\tau+ \frac{24 (1+\lam L) \tau}{1+\lam \mu}, \quad \forall j \in {\cal J}_3.
		\end{equation}
		
		\paragraph{To this end, we are in a position to prove \eqref{ineq:phi}.}
		Using Lemma \ref{lem:dD}, \eqref{ineq:tw}, and \eqref{ineq:grad}, we have
		\begin{equation}\label{ineq:Dd}
			D_h(w^j,\hat z) \stackrel{\eqref{ineq:D-d}}\le d_h(w^j,\hat z) + 4\left(L  + \frac{1}{\lam} \right) \sqrt{\frac{2\lam\tau}{1+\lam \mu}} \|w^j - \hat z\|, \quad \forall j\in\{1,\ldots,n\}.
		\end{equation}
		It clearly follows from Algorithm \ref{alg:STS} that
		\[
		|{\cal J}_1| > \frac{2n}{3}, \quad |{\cal J}_2| > \frac{2n}{3}, \quad |{\cal J}_3| > \frac{2n}{3}.
		\]
		By the pigeonhole principle, the intersection $\mathcal{J}_1 \cap \mathcal{J}_2 \cap \mathcal{J}_3$ must be nonempty. 
		Hence, it follows from \eqref{ineq:Dd}, \eqref{ineq:J3}, and \eqref{ineq:J1} that
		\[
		D_h(w^j,\hat z) 
		\le 3\tau+ \frac{48 (1+\lam L) \tau}{1+\lam \mu}, \quad \forall j\in {\cal J}_1 \cap {\cal J}_3.
		\]
		Consider an arbitrary index $j\in {\cal J}_1 \cap {\cal J}_2 \cap {\cal J}_3$. 
		Using the second inequality in \eqref{ineq:twoside1}, the above inequality, \eqref{ineq:J1}, and \eqref{ineq:J2}, we conclude
		\begin{align*}
			\phi^\lam(w^j) - \phi^\lam(\hat z) + \frac{1+\lam \mu}{\lam(2+\lam \mu)}\|z^j-\hat z\|^2 
			&\stackrel{\eqref{ineq:twoside1}}\le  \frac{1+\lam L}{2\lam}\|w^j-\hat z\|^2 + D_h(w^j,\hat z) + \frac{1+\lam \mu}{\lam(2+\lam \mu)}\|z^j-\hat z\|^2\\
			&\le \frac{1+\lam L}{1+\lam \mu} 9 \tau + \left(3\tau+ \frac{48 (1+\lam L) \tau}{1+\lam \mu} \right) + 9 \tau \\
			&\le 9 \kappa \tau + 3\tau+ 48 \kappa \tau + 9 \tau,
		\end{align*}
		where the last inequality is due to the fact that $(a+b)/(c+d)\le \max\{a/c, b/d\}$ for $a,b,c,d>0$.
		Therefore, \eqref{ineq:phi} follows.
		Finally, combining \eqref{ineq:E1} and \eqref{ineq:E21}, we complete the proof
		\[
		\Prob(E_1 \cap E_2) = \Prob( E_2 | E_1) \Prob(E_1) \ge \left(1-\exp\left( -\frac{n}{72}\right) \right)^2 \ge 1-2 \exp\left( -\frac{n}{72}\right).
		\]
	\end{proof}
	
	\section{High-Probability Result and Low Sample Complexity}\label{sec:result}
	
	This section provides the full analysis of Algorithm~\ref{alg:PPM} and establishes its two main guarantees: the high-probability convergence result (Theorem~\ref{thm:main}) and the sample complexity bound (Theorem~\ref{thm:sample}), which constitute the central contributions of this paper.

	The following proposition establishes a high-probability guarantee for a single iteration of Algorithm~\ref{alg:PPM}. 
	It yields a key recursive relation that serves as the foundation for proving the subsequent theorems.

	\begin{proposition}\label{prop:iter}
		Assuming that \eqref{def:tau} and \eqref{ineq:q} hold and considering a full iteration of Algorithm~\ref{alg:PPM}, then for every $k\ge 1$ and with probability at least $1-2 \exp\left( -n/72\right)$, we have
		\begin{equation}\label{ineq:iteration}
			\phi(\bar w_k)
			- \phi_* - \frac1{2\lam} \|x_*-\bar z_{k-1}\|^2 + \frac{1+\lam\mu}{\lam(4+\lam \mu)} \|x_*-\bar 
			z_k\|^2 \le 12 \tau + 57 \kappa \tau,
		\end{equation}
		where $x_*$ is the solution to \eqref{eq:ProbIntro} and $\tau$ is as in \eqref{eq:tau}.
	\end{proposition}
	
	\begin{proof}
		We restate Proposition \ref{prop:boost} using the notation in Algorithm \ref{alg:PPM} as: if \eqref{def:tau} and \eqref{ineq:q} hold, then with probability at least $1-2 \exp\left( -n/72\right)$, the pair $(\bar z_k, \bar w_k)$ generated in Step 2 of Algorithm~\ref{alg:PPM} satisfies
		\begin{equation}\label{ineq:assumption}
			\phi(\bar w_k)+\frac1{2\lam}\|\bar w_k-\bar z_{k-1}\|^2 -  \phi(\hat z_k) - \frac1{2\lam} \|\hat z_k-\bar z_{k-1}\|^2 + \frac{1+\lam \mu}{\lam(2+\lam \mu)}\|\hat z_k-\bar z_k\|^2 \stackrel{\eqref{ineq:phi}}\le \delta,
		\end{equation}
		where $\hat z_k$ is as in \eqref{eq:prox} and $\delta = 12 \tau + 57 \kappa \tau$.
		Using the fact that the proximal subproblem in \eqref{eq:prox} is $(\mu+1/\lam)$-strongly convex, we have
		\[
		\phi_*+ \frac1{2\lam} \|x_*-\bar z_{k-1}\|^2\ge \phi(\hat z_k)+ \frac1{2\lam} \|\hat z_k-\bar z_{k-1}\|^2 + \frac{1+\lam \mu}{2\lam}\|x_*-\hat z_k\|^2.
		\]
		This inequality and \eqref{ineq:assumption} then imply that with probability at least $1-2 \exp\left( -n/72\right)$,
		\begin{equation}\label{ineq:imply}
			\phi(\bar w_k)
			- \left(\phi_*+ \frac1{2\lam} \|x_*-\bar z_{k-1}\|^2\right) + \frac{1+\lam \mu}{2\lam}\|x_*-\hat z_k\|^2 + \frac{1+\lam \mu}{\lam(2+\lam \mu)}\|\bar z_k-\hat z_k\|^2 \le \delta.
		\end{equation}
		Furthermore, it follows from the triangle inequality and the Cauchy-Schwarz inequality that
		\[
		\|x_*-\bar z_k\|^2 \le [2 + (2+\lam \mu)] \left(\frac12 \|x_*-\hat z_k\|^2 + \frac{1}{2+\lam \mu}\|\bar z_k-\hat z_k\|^2\right),
		\]
		and hence that
		\[
		\frac{1+\lam\mu}{\lam(4+\lam \mu)} \|x_*-\bar z_k\|^2 \le \frac{1+\lam \mu}{2\lam}\|x_*-\hat z_k\|^2 + \frac{1+\lam \mu}{\lam(2+\lam \mu)}\|\bar z_k-\hat z_k\|^2.
		\]
		Combining the above inequality and \eqref{ineq:imply}, we conclude that \eqref{ineq:iteration} holds with probability at least $1-2 \exp\left( -n/72\right)$.
	\end{proof}
	
	We are now prepared to state the main theorem, which establishes the high-probability guarantee of Algorithm~\ref{alg:PPM}.

	\begin{theorem}\label{thm:main}
		Assume that $\lam \mu \ge 3$, \eqref{def:tau}, and \eqref{ineq:q} hold. For given $\varepsilon>0$ and $p\in (0,1)$, consider the solution sequence $\{\bar w_k\}_{k=1}^K$ generated by Algorithm~\ref{alg:PPM}, if the input triple $(K,I,n)$ satisfies
		\begin{equation}\label{eq:input}
			K= {\cal O}\left(\log \frac{1}{\varepsilon}\right), \quad I = {\cal O}\left(\max\left\{\kappa \log\frac{\kappa}{\varepsilon},\frac{\kappa \sigma^2}{\mu \varepsilon}\right\}\right), \quad n={\cal O}\left(\log \frac1p \right),
		\end{equation}
		where $\kappa=L/\mu>1$ is the condition number,
		then we have
		\begin{equation}\label{ineq:final}
			\Prob\left(\min_{1\le k \le K} \phi(\bar w_k)- \phi_* \le \varepsilon\right) \ge 1-p.
		\end{equation}
	\end{theorem}
	
	\begin{proof}
		Define
		\begin{equation}\label{def:recur}
			a_k = \phi(\bar w_k)- \phi_*, \quad b_k = \frac1{2\lam} \|x_*-\bar z_k\|^2, \quad \theta = \frac{2+2\lam \mu}{4+\lam \mu}, \quad \delta = 12 \tau + 57 \kappa \tau.
		\end{equation}
		Then, using Proposition \ref{prop:iter}, we have with probability at least $1-2 \exp\left( -n/72\right)$,
		\begin{equation}\label{eq:easyrecur1}
			a_k \stackrel{\eqref{ineq:iteration}}\le b_{k-1} - \theta b_k + \delta.
		\end{equation}
		Multiplying \eqref{eq:easyrecur1} by $\theta^{k-1}$ and summing the resulting inequality from $ k=1 $ to $ K $, we have
		\begin{equation}\label{ineq:sum}
			\sum_{k=1}^K \theta^{k-1} a_k \le  \sum_{k=1}^K \theta^{k-1} (b_{k-1} - \theta b_k + \delta) = b_0 - \theta^K b_K + \sum_{k=1}^K \theta^{k-1} \delta, 
		\end{equation}
		with probability at least
		\begin{equation}\label{ineq:p}
			\left(1-2\exp\left( -\frac{n}{72}\right) \right)^K \ge 1-2K\exp\left( -\frac{n}{72}\right).
		\end{equation}
		Dividing \eqref{ineq:sum} by $\sum_{k=1}^K \theta^{k-1}$ and using \eqref{def:recur}, we have
		\begin{equation}\label{ineq:min}
			\min_{1\le k \le K} \phi(\bar w_k)- \phi_* = \min_{1\le k \le K} a_k \le \frac{b_0}{\sum_{k=1}^K \theta^{k-1}} + \delta \le \frac{\|\bar z_0-x_*\|^2}{2\lam \sum_{k=1}^K \theta^{k-1}} + 12 \tau + 57 \kappa \tau,
		\end{equation}
		with probability as in \eqref{ineq:p}.
		It follows from the assumption that $\lam \mu \ge 3$ and the definition of $\theta$ in \eqref{def:recur} that $\theta \ge 8/7$, which together with \eqref{ineq:min}, implies that
		\begin{equation}\label{ineq:min1}
			\min_{1\le k \le K} \phi(\bar w_k)- \phi_* \le \frac{\|\bar z_0-x_*\|^2}{14\lam \left[\left(\frac{8}{7}\right)^K -1\right]} + 12 \tau + 57 \kappa \tau,
		\end{equation}
		with probability as in \eqref{ineq:p}.
		It is clear that if the first condition in \eqref{eq:input} holds, i.e.,
		\begin{equation}\label{eq:outer}
			K= {\cal O}\left(\log \frac{1}{\varepsilon}\right),
		\end{equation}
		then the first term on the right-hand side of \eqref{ineq:min1} satisfies
		\begin{equation}\label{ineq:error}
			\frac{\|\bar z_0-x_*\|^2}{14\lam \left[\left(\frac{8}{7}\right)^K -1\right]} = {\cal O}(\varepsilon).
		\end{equation}
		Since the last term $57 \kappa \tau$ in \eqref{ineq:min1} dominates the second term $12\tau$, in view of \eqref{eq:tau}, we only need to derive a bound on $I$ such that
		\begin{equation}\label{ineq:tau}
			\alpha^I \left(\sigma D + \frac{L D^2}2\right) + \frac{\lam \sigma^2}{I} \le \frac{\varepsilon}{\kappa}.
		\end{equation}
		Using the above inequality and the fact that $\alpha\le e^{\alpha-1}$, we know the iteration count of Algorithm~\ref{alg:PSS} is
		\[
		I = {\cal O}\left(\max\left\{\frac{1}{1-\alpha}\log\frac{\kappa}{\varepsilon},\frac{\lam \kappa \sigma^2}{\varepsilon}\right\}\right).
		\]
		It thus follows from \eqref{def:tau} and $\lam \mu \ge 3$ that
		\begin{equation}\label{eq:inner}
			I = {\cal O}\left(\max\left\{\kappa \log\frac{\kappa}{\varepsilon},\frac{\kappa \sigma^2}{\mu \varepsilon}\right\}\right),
		\end{equation}
		which is the second condition in \eqref{eq:input}.
		Now, assuming \eqref{eq:outer} and \eqref{eq:inner} hold, then we know \eqref{ineq:error} and \eqref{ineq:tau} should also hold. Putting \eqref{ineq:min1}, \eqref{ineq:error}, and \eqref{ineq:tau} together, we have
		\[
		\Prob\left(\min_{1\le k \le K} \phi(\bar w_k)- \phi_* \le \varepsilon\right) \ge 1-2K\exp\left(-\frac{n}{72}\right) \stackrel{\eqref{eq:outer}}\approx 1-\exp\left(-\frac{n}{72}\right).
		\]
		Finally, the above inequality and the last condition in \eqref{eq:input} imply that \eqref{ineq:final} holds.
	\end{proof}
	
	We now arrive at the final main result, which establishes the low sample complexity of Algorithm~\ref{alg:PPM}.
	
	\begin{theorem}\label{thm:sample}
		For given $\varepsilon>0$ and $p\in (0,1)$, to find a solution $\bar w \in \dom h$ by Algorithm~\ref{alg:PPM} such that
		\[
		\Prob\left(\phi(\bar w)- \phi_* \le \varepsilon\right) \ge 1-p,
		\]
		we need ${\cal O}\left(\log\frac1p \log\frac1\varepsilon\right)$ calls to Algorithm~\ref{alg:PSS} and ${\cal O}(\log \frac1\varepsilon)$ calls to Algorithm~\ref{alg:PB}.
		Moreover, the sample complexity of stochastic gradients in Algorithm~\ref{alg:PPM} is
		\begin{equation}\label{eq:total}
			{\cal O}\left(\max\left\{\kappa \log\frac{\kappa}{\varepsilon},\frac{\kappa \sigma^2}{\mu \varepsilon}\right\} \log \frac1p \log \frac1{\varepsilon} \right).
		\end{equation}
	\end{theorem}
	
	\begin{proof}
		It is clear that Algorithm~\ref{alg:PPM} requires $nK$ PSS oracles (Algorithm~\ref{alg:PSS}) and $K$ PB oracles (Algorithm~\ref{alg:PB}).
		Using Theorem \ref{thm:main}, we know that the numbers of calls to PSS and PB oracles are ${\cal O}\left(\log\frac1p \log\frac1\varepsilon \right)$ and ${\cal O}(\log \frac1\varepsilon)$, respectively.
		From Algorithm~\ref{alg:PSS}, we know that each PSS oracle has $I+1$ iterations and each iteration takes one stochastic gradient sample. Therefore, PSS takes $nK(I+1)$ samples in total.
		From Algorithm~\ref{alg:PB}, we know that each PB oracle queries an RGE subroutine and each RGE takes $nq$ stochastic gradient samples. Therefore, PB takes $nKq$ samples in total.
		Using \eqref{eq:tau}, \eqref{ineq:q}, and the fact that $\lam \mu = \Omega(1)$, we have
		\[
		q \stackrel{\eqref{ineq:q}}={\cal O} \left( \frac{\sigma^2}{\lam L^2 \tau} \right) \stackrel{\eqref{eq:tau}}= {\cal O} \left(\frac{I}{\lam^2 L^2}\right) = {\cal O} \left(\frac{I}{\kappa^2}\right).
		\]
		Therefore, the total sample complexity is ${\cal O} (nKI)$, which is \eqref{eq:total} in view of \eqref{eq:input}.
	\end{proof}
	
	\section{Conclusions}\label{sec:conclusion}
	
	In this paper, we studied the stochastic convex composite optimization problem \eqref{eq:ProbIntro} with the goal of establishing high-probability guarantees, namely $\Prob(\phi(x)-\phi_* \le \varepsilon) \ge 1-p$ for given $\varepsilon>0$ and $p\in (0,1)$. 
	Existing approaches that achieve ${\cal O}(\log(1/p))$ dependence on $p$ in the sample complexity typically rely on restrictive noise assumptions such as sub-Gaussian tails. 
	Our contribution is to show that such guarantees are attainable under the much milder bounded-variance assumption by designing SPPM, a stochastic proximal point scheme with a constant prox stepsize. 
	The key ingredients are the subroutine PSS, which both approximates the proximal subproblem \eqref{eq:prox} and reduces variance, and the subroutine PB, which ensures that each selected iterate is sufficiently accurate with high probability. 
	Together these yield the main results of the paper, namely Theorem~\ref{thm:main} (high-probability guarantee) and Theorem~\ref{thm:sample} (low sample complexity).
	
	We close by discussing several directions for future work. 
	First, while our analysis shows that SPPM achieves logarithmic dependence on $1/p$, its overall complexity still carries an overhead proportional to the condition number $\kappa$. 
	A natural question is whether an accelerated variant of SPPM could reduce this dependence to $\sqrt{\kappa}$, in analogy with accelerated gradient methods. 
	Although acceleration is often believed to amplify noise, it is worth investigating whether a restarted acceleration scheme, interpreted through the proximal point framework (see~\cite{liang2025unifying}) in the same way that we have analyzed PSS here, could instead lead to improved variance reduction. 
	Second, our current theorems require certain assumptions on problem parameters to state the guarantees (such as $\lam \mu \ge 3$ in Theorem~\ref{thm:main}). 
	An important extension would be to weaken or remove these assumptions in order to design methods that adapt automatically to the problem structure. 
	Finally, it would be interesting to explore variants of SPPM with a variable prox stepsize $\lambda_k$, which may further expand the flexibility of the framework and sharpen the theoretical guarantees.
	
	\bibliographystyle{plain}
	\bibliography{ref}
	
	\appendix
	
	\section{Technical Results}\label{sec:tech}
	
	This section collects several useful technical results used throughout our analysis.
	
	\begin{lemma}\label{lem:dh}
		Recall $d_h(\cdot,\cdot)$ defined in Step 4 of Algorithm \ref{alg:PB}. Then $d_h(\cdot,\cdot)$ is a metric.
	\end{lemma}
	
	\begin{proof}
		Recall the notation $\bar s^\lam(\tilde w):=\bar s(\tilde w) + (\tilde w - \bar z)/\lam$ used in Section \ref{sec:PB}.
		Thus, it follows from the definition of $d_h(x, y)$ in Step 4 of Algorithm \ref{alg:PB} that
		\[
		d_h(x, y)=|h(x)-h(y)+\langle\bar s^\lam(\tilde w), x-y\rangle|.
		\]
		Since $d_h(x, y)\ge 0$ and $d_h(x, x)=0$, nonnegativity holds.
		It is clear that $d_h(x, y)=d_h(y, x)$. Hence, symmetry also holds.
		Finally, we also have the triangle inequality $d_h(x,y)+d_h(y,z)\ge d_h(x,z)$, since
		\begin{align*}
			d_h(x, z)
			&=|h(x)-h(y) + h(y)-h(z) +\langle\bar s^\lam(\tilde w), x-y + y -z\rangle| \\
			&\le|h(x)-h(y)+\langle\bar s^\lam(\tilde w), x-y\rangle| + |h(y)-h(z)+\langle\bar s^\lam(\tilde w), y-z\rangle| \\
			&= d_h(x,y)+d_h(y,z).
		\end{align*}
	\end{proof}
	
	\begin{lemma}\label{lem:dD}
		Recall $D_h(\cdot,\hat z)$ and $d_h(\cdot,\hat z)$ given in \eqref{def:Bregman} and \eqref{eq:dh}, respectively. Then, for every $x\in \dom h$, we have
		\begin{align}
			d_h(x,\hat z) &\le D_h(x,\hat z) + \left(\|\bar s(\tilde w) - \nabla f (\hat z)\| + \frac{1}{\lam} \|\tilde w - \hat z\|\right) \|x - \hat z\|, \label{ineq:d-D} \\
			D_h(x,\hat z) &\le d_h(x,\hat z) + \left(\|\bar s(\tilde w) - \nabla f (\hat z)\| + \frac{1}{\lam} \|\tilde w - \hat z\|\right) \|x - \hat z\|, \label{ineq:D-d}
		\end{align}
		where $\hat z$ is as in \eqref{eq:sub}, and $\bar s(\tilde w)$ and $\tilde w$ are defined in Step 3 of Algorithm \ref{alg:PB}.
	\end{lemma}
	
	\begin{proof}
		Following the notation in Section \ref{sec:PB}, we note that
		\[
		\bar s^\lam(\tilde w)=\bar s(\tilde w) + \frac{\tilde w- \bar z}{\lam}, \quad \nabla f^\lam(\hat z) = \nabla f(\hat z) + \frac{\hat z - \bar z}{\lam}.
		\]
		Using \eqref{eq:dh}, the above relations, and the triangle inequality, we have for every $x\in \dom h$,
		\begin{align*}
			d_h(x,\hat z)
			&\stackrel{\eqref{eq:dh}}=|h(x) - h(\hat z) + \inner{\nabla f^\lam(\hat z)}{x-\hat z} + \inner{\bar s^\lam(\tilde w) - \nabla f^\lam(\hat z)}{x - \hat z}| \\
			&\le|h(x) - h(\hat z) + \inner{\nabla f^\lam(\hat z)}{x-\hat z}| + |\inner{\bar s(\tilde w) - \nabla f(\hat z)}{x - \hat z}| + \frac{1}{\lam}|\inner{\tilde w - \hat z}{x - \hat z}|.
		\end{align*}
		It follows from the definition of $D_h$ in \eqref{def:Bregman} and the Cauchy-Schwarz inequality that \eqref{ineq:d-D} holds.
		The proof of \eqref{ineq:D-d} follows similarly.
	\end{proof}
	
	\begin{lemma}
		Recall $D_h(\cdot,\hat z)$ and $\phi^\lam(\cdot)$ given in \eqref{def:Bregman} and \eqref{def:flam}, respectively. Then, for every $x\in \dom h$, we have
		\begin{equation}\label{ineq:twoside1}
			\frac{1+\lam \mu}{2\lam}\|x-\hat z\|^2 + D_h(x,\hat z) \\
			\le  \phi^\lam(x) - \phi^\lam(\hat z) \\
			\le  \frac{1+\lam L}{2\lam}\|x-\hat z\|^2 + D_h(x,\hat z).
		\end{equation}
	\end{lemma}
	\begin{proof}
		It follows from \eqref{ineq:twoside} and the definition of $f^\lam$ in \eqref{def:flam} that for every $x \in \dom h$,
		\[
		\frac{1+\lam \mu}{2\lam}\|x-\hat z\|^2 \le  f^\lam(x) - f^\lam(\hat z) - \inner{\nabla f^\lam(\hat z)}{x-\hat z} 
		\le \frac{1+\lam L}{2\lam}\|x-\hat z\|^2.
		\]
		Adding $D_h(x,\hat z)$ to the above inequality and using the fact that $\phi^{\lam}(\cdot)=f^{\lam}(\cdot) + h(\cdot)$ completes the proof.
	\end{proof}
	
	\begin{lemma}\label{lem:Hoeffding}
		Consider a sequence of independent events $\{A_j\}_{j=1}^n$ where each event occurs with probability at least $3/4$.
		Define index set ${\cal J}:=\left\{j\in\{1,\ldots,n\}: A_j \text{ occurs}\right\}$.
		Then, we have
		\begin{equation}\label{ineq:prob1}
			\Prob\left(|{\cal J}|> \frac {2n}3\right)\ge 1-\exp\left( -\frac{n}{72}\right).
		\end{equation}
	\end{lemma}
	
	\begin{proof}
		Recall that Hoeffding's inequality states that: let $X_1, \ldots, X_n$ be independent random variables such that $a_j\le X_j \le b_j$ and $S_n = X_1+\ldots+X_n$, then for all $t>0$,
		\begin{equation}\label{ineq:Hoeffding}
			\Prob\left(S_n - \E[S_n]\le -t\right) \le \exp\left(\frac{-2t^2}{\sum_{j=1}^n(b_j-a_j)^2}\right).
		\end{equation}
		Define the random indicator variable $X_j$ associated with $A_j$ as follows
		\begin{align*}
			X_j :=  \left\{\begin{array}{ll}
				1, & \text { if } A_j \text{ occurs}, \\ 
				0,  & \text { otherwise}.
			\end{array}\right.
		\end{align*}	
		Clearly, $X_j$'s are independent. For every $j=1,\ldots,n$, we have $a_j=0$, $b_j=1$, and
		\begin{equation}\label{ineq:E}
			\E[X_j] = \Prob(X_j=1) = \Prob(A_j \text{ occurs}) \ge \frac34.
		\end{equation}
		It follows Hoeffding's inequality with $t=n/12$ that
		\[
		\Prob\left(S_n \ge \frac{2n}{3}\right) \stackrel{\eqref{ineq:E}}\ge \Prob\left(S_n - \E[S_n] \ge -\frac{n}{12}\right) \stackrel{\eqref{ineq:Hoeffding}}\ge 1- \exp\left(-\frac{n}{12}\right).
		\]
		Therefore, \eqref{ineq:prob1} immediately follows.
	\end{proof}

	\section{Second Tertile Selection}\label{sec:sts}
	
	This section presents and analyzes the subroutine STS used in Algorithm \ref{alg:PB}. Algorithm~\ref{alg:STS} below is motivated by Algorithm 8 of \cite{davis2021low}.
	
	\begin{algorithm}[H]
		\caption{Second Tertile Selection, STS$(Z, d(\cdot,\cdot))$}
		\begin{algorithmic}
			\STATE \textbf{Input:} A set of points $Z=\{z^1,\ldots,z^n\} \subset \dom h$ and a metric $d(\cdot,\cdot)$ on $\dom h$.
			\FOR{$j = 1, \ldots, n$}
			\STATE {\bf Step 1.} Compute $\rho_j=\min \left\{\rho > 0:\left|B_{\rho,d}\left(z^j\right) \cap Z\right|>2n/3\right\}$;
			\ENDFOR
			\STATE {\bf Step 2.} Compute the second tertile $\bar \rho$ of $\left(\rho_1, \ldots, \rho_n\right)$.
			\STATE \textbf{Output:} ${\cal J}=\left\{j\in\{1,\ldots,n\}: \rho_j \le \bar \rho\right\}$.
		\end{algorithmic}\label{alg:STS}
	\end{algorithm}
	
	Here, $B_{\rho,d}\left(z\right)$ denotes a $d$-metric ball centered at $z$ with radius $\rho$, that is, $B_{\rho,d}\left(z\right)= \{x: d(x,z) \le \rho\}$.
	Note that STS takes as input a collection of points $\{z^j\}_{j=1}^n$ and a metric $d(\cdot,\cdot)$. In particular, we use $d_2(\cdot,\cdot)$ and $d_h(\cdot,\cdot)$ in Algorithm \ref{alg:PB}. The output ${\cal J}$ of STS is an index set with a cardinality of at least $2n/3$.
	
	The following lemma outlines the key property of Algorithm \ref{alg:STS}, specifically its ability to maintain proximity for most input points in $Z$ to any point. Typically, this lemma is employed in conjunction with Lemma~\ref{lem:Hoeffding} to boost the probability from low to high confidence, impacting the overall quality by at most a constant factor.
	
	\begin{lemma}\label{lem:3eps}
		Let $d(\cdot,\cdot)$ be a metric on $\dom h$. Consider a collection of points $Z=\{z^1,\ldots,z^n\}$ and a point $\tilde z \in \dom h$ satisfying 
		\begin{equation}\label{cond:card}
			|B_{\epsilon,d}(\tilde z) \cap Z|>\frac{2n}{3}
		\end{equation}
		for some $\epsilon>0$. 
		Then, output index set ${\cal J}$ of Algorithm \ref{alg:STS}) satisfies
		\begin{equation}\label{ineq:3eps}
			d(z^j,\tilde z) \le 3 \epsilon, \quad \forall j \in {\cal J}.
		\end{equation}
	\end{lemma}
	
	\begin{proof} 
		Consider two arbitrary points $z^i, z^j \in B_{\epsilon,d}(\tilde z)$. Since $d(\cdot,\cdot)$ is a metric, the triangle inequality holds, and hence
		\[
		d(z^i,z^j) \le d(z^i,\tilde z) + d(\tilde z,z^j) \le 2\epsilon.
		\]
		Thus, for any $z^j \in B_{\epsilon,d}(\tilde z)$ fixed,  we have $|B_{2\epsilon,d}(z^j) \cap Z|>\frac{2n}{3}$ and consequently $\rho_j\le 2\epsilon$ (see Step~1 of Algorithm~\ref{alg:STS}). Note that there are at least $2n/3$ such $z^j$.
		This observation further implies that the second tertile $\bar \rho \le 2\epsilon$ (see Step~2 of Algorithm~\ref{alg:STS}).
		
		Now, we consider an arbitrary index $j \in {\cal J}$.
		Since both $B_{\epsilon,d}(\tilde z)$ and $B_{\rho_j,d}(z^j)$ contain at least $2n/3$ points, by the pigeonhole principle, there must exist a point $z$ at the intersection $ B_{\epsilon,d}(\tilde z) \cap B_{\rho_j,d}(z^j)$. Using the triangle inequality, we conclude that 
		\[
		d(\tilde z,z^j) \le d(\tilde z, z) + d(z,z^j) \le \epsilon + \rho_j \le \epsilon + 2\epsilon = 3\epsilon.
		\]
		Since the above inequality holds for any $j\in {\cal J}$, \eqref{ineq:3eps} immediately follows.
	\end{proof}
	
	\section{Robust Gradient Estimation}\label{sec:rge}
	
	This section presents and analyzes the subroutine RGE used in Algorithm \ref{alg:PB}. Algorithm~\ref{alg:RGE} below is motivated by \cite{davis2021low}.
	
	\begin{algorithm}[H]
		\caption{Robust Gradient Estimation, RGE$(x,n,q)$}
		\begin{algorithmic}
			\STATE \textbf{Input:} A point $x\in \dom h$ and integers $n, q\ge 1$.
			\FOR{$j = 1, \ldots, n$}
			\STATE {\bf Step 1.} Generate $q$ independent stochastic gradients $s(x,\xi_j^1),\ldots,s(x,\xi_j^q)$ and compute $\bar s_j(x) = \frac{1}{q} \sum_{i=1}^q s(x,\xi_j^i)$;
			\ENDFOR
			\STATE {\bf Step 2.} Call oracle $\mathcal{J}:=\operatorname{STS}(S(x), d_2(\cdot,\cdot))$ where $S(x)=\{\bar s_1(x),\ldots,\bar s_n(x)\}$.
			\STATE \textbf{Output:} $\bar s_{j^*}(x)$ for an arbitrary index $j^*\in {\cal J}$.
		\end{algorithmic}\label{alg:RGE}
	\end{algorithm}
	
	In contrast to PSS, i.e., Algorithm \ref{alg:PSS}, RGE achieves variance reduction by using a batch of $q$ independent stochastic gradients. The following lemma presents a concentration inequality of the stochastic gradient estimate.
	
	\begin{lemma}\label{lem:RGE}
		In Algorithm \ref{alg:RGE}, if $q\ge 4\sigma^2/\delta^2$ where $\sigma$ is as in Assumption (A2) and $\delta>0$ is some scalar, then for any $x\in \dom h$ and $n\ge 1$, the output $\bar s_{j^*}(x)$ satisfies
		\begin{equation}\label{ineq:RGE}
			\Prob\left( \|\bar s_{j^*}(x)-\E[\bar s_{j^*}(x)]\| \le 3\delta \right) \ge 1-\exp\left(-\frac{n}{72}\right).
		\end{equation}
	\end{lemma}
	
	\begin{proof}
		We first note that $\nabla f(x) = \E[\bar s_{j^*}(x)]$ due to Assumption (A1), and hence use $\nabla f(x)$ throughout the proof for simplicity.
		Using Assumptions (A1) and (A2), we know that the estimate $\bar s_j(x)$ has a variance reduction by a factor of $q$, that is,
		\begin{align*}
			\E\left[\|\bar s_j(x)-\nabla f(x)\|^2\right]&= \E\left[\left\|\frac{1}{q}\sum_{i=1}^q s(x,\xi_j^i)-\nabla f(x)\right\|^2\right]
			= \frac{1}{q^2} \sum_{i=1}^q \E\left[\left\| s(x,\xi_j^i)-\nabla f(x)\right\|^2\right] 
			\le \frac{\sigma^2}{q}.
		\end{align*}
		It follows from the assumption that $q\ge 4\sigma^2/\delta^2$ and Markov's inequality that
		\[
		\Prob\left( \|\bar s_j(x)-\nabla f(x)\|^2 \ge \delta^2 \right) \le \frac{\sigma^2}{q \delta^2} \le \frac{1}{4}.
		\]
		Hence, for every $j=1,\ldots,n$, we have
		\[
		\Prob\left( \|\bar s_j(x)-\nabla f(x)\| \le \delta \right) \ge \frac{3}{4}.
		\]
		Using Lemma \ref{lem:Hoeffding} with $A_j=\{\|\bar s_j(x)-\nabla f(x)\| \le \delta\}$, we have
		\[
		\Prob\left(\left|B_{\delta,d_2}(\nabla f(x)) \cap S(x)\right|>\frac{2n}{3}\right) \ge 1-\exp\left(-\frac{n}{72}\right).
		\]
		This means condition \eqref{cond:card} in Lemma \ref{lem:3eps} holds with probability at least $1-\exp\left(-\frac{n}{72}\right)$.
		Since we call the oracle ${\cal J}=\operatorname{STS}(S(x), d_2(\cdot,\cdot))$ in Step 2, it follows from Lemma \ref{lem:3eps} that for every $j \in {\cal J}$,
		\[
		\Prob\left( \|\bar s_{j}(x)-\nabla f(x)\| \le 3\delta \right) \ge 1-\exp\left(-\frac{n}{72}\right).
		\]
		Therefore, \eqref{ineq:RGE} holds for any $j^*\in {\cal J}$.
	\end{proof}
	
\end{document}